\documentclass[11pt]{article}

\usepackage{amsfonts,amsmath,amsfonts,amssymb,graphicx,graphics, bbm}
\usepackage{authblk}
\usepackage{enumerate}
\usepackage{color}
\usepackage{algorithm}
\usepackage{algpseudocode}

\newtheorem{theorem}{Theorem}[section]
\newtheorem{lemma}[theorem]{Lemma}

\newtheorem{proposition}[theorem]{Proposition}
\newtheorem{corollary}[theorem]{Corollary}

\newenvironment{proof}[1][Proof]{\noindent\textbf{#1.} }{\ \rule{0.5em}{0.5em}}

\DeclareMathOperator{\rank}{rank}
\DeclareMathOperator*{\argmin}{arg\,min}

\def\E{\mathbb{E}}
\def\P{\mathbb{P}}

\title{Two-sided Matrix Regression}

\author[1]{Nayel Bettache} 
\author[1]{Cristina Butucea}  
\affil[1]{CREST, ENSAE, Institut Polytechnique de Paris, 
5 avenue Henry Le Chatelier, 91120 Palaiseau, France}

\begin{document}

\maketitle

\begin{abstract}%
The two-sided matrix regression model $Y = A^*X B^* +E$ aims at predicting $Y$ by taking into account both linear links between column features of $X$, via the unknown matrix $B^*$, and also among the row features of $X$, via the matrix $A^*$. We propose low-rank predictors in this high-dimensional matrix regression model via rank-penalized and nuclear norm-penalized least squares. Both criteria are non jointly convex; however, we propose explicit predictors based on SVD and show optimal prediction bounds. We give sufficient conditions for consistent rank selector. We also propose a fully data-driven rank-adaptive procedure. Simulation results confirm the good prediction and the rank-consistency results under data-driven explicit choices of the tuning parameters and the scaling parameter of the noise.
\end{abstract}


\section{Introduction}

Supervised learning is often performed on large data bases. Matrix regression assumes that the data $Y$ can be well explained by a set of features given by the columns of the matrix $X$ and linear combinations of these columns. It is often the case in real-life that the rows of $Y$ can be explained by linear combinations of the rows of $X$. 

For example, economic data store economic indicators as column features and countries as rows. Such a matrix is usually explained by a smaller matrix roughly containing a smaller number of countries (representatives of groups of geographically or economically close countries) and a few economic features or some factors produced out of all these indicators. We would like to predict a larger number of indicators for a larger number of countries, {\it i.e.} $Y$ a $n\times p$ matrix, using the features $X$ a $m \times q$ matrix.\\
Recommendation systems want to predict the opinion of $n$ clients concerning $p$ items. We can use publicly available data on a number $m$ of different groups of clients and their affinity to a number $q$ of large categories of items in order to predict by evaluating the client's correlation to the prescribed groups in the population and the item's weight in its category. We may include a multiple-label situation where the items belonging to a main category are also related to other categories.\\
Other examples can be given for meteorological data, medical or pharmaceutical data and so on.

{\bf Model.} We observe the matrix $Y \in \mathbb{R}^{n \times p}$ and a design matrix $X \in \mathbb{R}^{m \times q}$ related via the {\bf two-sided matrix regression (2MR)} model involving two parameter matrices $A^* \in \mathbb{R}^{n \times m}$ and $B^* \in \mathbb{R}^{q \times p}$:
\begin{equation} \label{eq: model matrix}
    Y = A^* X B^* + E,
\end{equation}
where the noise matrix $E$ is assumed to have independent centered $\sigma-$sub-Gaussian entries. 

The 2MR model encompasses known models like, {\it e.g.} matrix regression and matrix factorisation.
Indeed, if $n=m$ and $A^*$ is the identity, the matrix model \eqref{eq: model matrix} becomes the (one-sided) {\it matrix regression} (MR) model $Y = X B^* + E$, see \cite{obozinski}, \cite{bunea}, \cite{negahban}.\\
Assume now that $m=q$ and that the design matrix $X$ is the identity matrix of rank $m$ smaller than both $n$ and $p$. Our model becomes a {\it factorisation model} of the signal $M^* = A^* B^*$ observed with noise. The idea is to recover a low-rank structure generating the observed data. In \cite{klopp} the authors have considered structured factorisation of the signal under assumptions that the rows of $A^*$ and the columns of $B^*$ have a common sparsity parameter and $X$, which they do not observe, has a much smaller dimension than $Y$.

The 2MR model \eqref{eq: model matrix} is strongly related to other models, but we argue that it cannot be reduced to these other models of a different nature. Indeed, note that the entry $Y_{ij}$ of the matrix $Y$ can be written
$$
Y_{ij} = {\rm Tr} (X \cdot B^*_{\cdot, j} A^*_{i,\cdot}) + E_{ij},
$$
for any $i$ in $[n]$, where $[n] = \{1, \ldots, n\}$, and for any $j$ in $[p]$. Thus every entry $Y_{ij}$ brings information through the same design matrix $X$ on the rank 1 matrix $B^*_{\cdot, j} A^*_{i,\cdot}$. This is unlike {\it the trace-regression model} or the more general {\it matrix completion} studied by \cite{rohde2011estimation}, \cite{koltchinskii}, where a different design matrix brings information on the parameter matrix $B^*A^*$. \\
Another way of writing model \eqref{eq: model matrix} is in the form of {\it vector regression model}, by stacking the columns of matrices $Y$, $X$ and $E$ into $vec(Y)$, $vec(X)$ and $vec(E)$, respectively, to get
\begin{equation} \label{eq: vec form model}
    vec(Y)^\top = vec(X)^\top \cdot A^\top \otimes B + vec(E)^\top,
\end{equation}
where $\otimes$ denotes the tensor product of two matrices. Under this relation, we predict a row vector of size $np$ using a row vector of size $mq$ (the matrix of features has rank 1) via a parameter of size $(mq) \times (np)$ which cannot go well unless the structure of $A$ and $B$ is trivial. This approach cannot take into account the matrix structure of the features, of the matrices $A^*$, $B^*$, and it gives poor results on that account.

This model has been introduced in time series by \cite{RongChen} as the {\it auto-regressive matrix-valued model} of order 1, MAR(1), $Y_t = A^* Y_{t-1} B^* + E_t$, observed at times $t$ in $[T]$. In this case $A^*$ and $B^*$ are squared matrices with spectral radii strictly less than 1 in order to ensure stability of the time series ($X_t$ is thus stationary and causal). The authors propose three estimation methods: first, they use the vector form analogous to \eqref{eq: vec form model}, stack the $T$ lines of $vec(Y_t)^\top$ and they use the nearest Kronecker product (NKP) problem to give estimators of $A^*$ and $B^*$ out of the global least squares estimator of ${A^*}^\top \otimes B^*$; then, their next method minimizes the least squares over $A$ and $B$
$$
\min_{A,B} \, \frac 1T \sum_{t=1}^T \|Y_t - A Y_{t-1} B\|_F^2,
$$
by a sequential procedure minimizing over $A$ for fixed given $B$, then over $B$ for fixed $A$, and iterating; finally, they give an MLE procedure over $A$ and $B$ under a particular structure of the covariance matrix of $E$ and proceed also sequentially. Theoretical results state the asymptotic normality as $T$ tends to infinity, for fixed dimensions. However, the first procedure is cumbersome as the estimated matrix is very large, while the other two procedures are based on non-convex minimization without theoretical guarantees as to the limit points of the algorithm. \\
Least squares and MLE estimators with AIC and BIC penalties have been numerically studied by \cite{hsu} of a more general time series model 
$$
Y_t = \sum_{\ell = 1}^L A_\ell Y_{t-\ell} B_\ell + E_t, \quad t=1,\ldots, T,
$$
which is treated as $Y_t = A^* X_{t} B^* + E_t$, where $X_{t}$ is the block diagonal matrix containing the $L-$past observed matrices $Y_{t-1}, \ldots, Y_{t-L}$ and $A^* = (A_1,\ldots,A_L)$ and $B^* = (B_1^\top,\ldots,B_L^\top)^\top$ are the concatenated matrices in the previous equation. \\
Thus, our paper is motivated by the need to deal with high-dimensional data and finite (non-asymptotic) time (say $T=1$) in order to provide theoretical guarantees for prediction.

{\bf Contributions. } We show in Section~\ref{sec: rank} that by using the SVD of matrices $Y = U_Y\Sigma_Y V_Y^T$ and $X = U_X \Sigma_X V_X^T$, the least squares procedure can be reduced to fitting predictors of the form $A_0 \Sigma_X B_0$ to the diagonal matrix $\Sigma_Y$ with explicit relations between $A_0,\, B_0$ and $A, B$. There is a natural choice of predictors of $A_0$ and of $B_0$ under diagonal form. We study these predictors for given ranks $r$ and that we transform back into the original space of $Y$ without loss of prediction rate. Then we give a data-dependent rank selector and show that the predictors associated to it attain optimal bounds. We give sufficient conditions so that the rank selector is consistent. Finally, we slightly modify the procedure to be free of the parameter $\sigma$ of the noise and show new upper bounds in this case.  In Section~\ref{sec: nuclear}, we study the nuclear norm penalized least squares and show it attains the optimal bounds too. All proofs are in a dedicated section in the Appendix. Finally, we illustrate in Section~\ref{sec: numerics} via numerical simulations the excellent prediction results of these fast running, explicit predictors.

{\bf Notations.} For any integers $n$ and $m$ we denote $n \wedge m$ for the minimum between $n$ and $m$ and $n \vee m$ for the maximum between $n$ and $m$. For any matrix $M $ of size ${n \times m}$ and rank $r_M$, we denote its singular value decomposition (SVD) by $M = U_M \Sigma_M V_M^\top$, where $U_M$ belongs to $\mathcal{O}_n$ - the set of orthogonal matrices of size $n\times n$, $V_M$ belongs to $\mathcal{O}_m$ and 
$$
\Sigma_M = Diag_{n,m}(\sigma_k(M), \, 1\leq k \leq {r_M}).
$$ 
Note that $\sigma_1(M), \ldots, \sigma_{r_M}(M)$ are the positive singular values of $M$ listed in decreasing order, and the $n \times m$ diagonal matrix $Diag_{n,m}(\sigma_k(M), \, 1\leq k \leq {r_M})$ has diagonal entries in the list and 0 elsewhere. 
Furthermore, denote $\|M\|_F^2 = \sum \limits_{k=1}^{n \wedge m} \sigma_k(M)^2$ its Frobenius norm, $\|M\|_{(2, q)}^2 = \sum \limits_{k=1}^{q} \sigma_k(M)^2$ its Ky-Fan $(2, q)$ norm, $\|M\|_{op} = \sigma_1(M)$ its operator norm, $\|M\|_*=\sum \limits_{k=1}^{n \wedge m} \sigma_k(M)$ its nuclear norm, $M^\dagger$ its Moore-Penrose inverse, $r_M$ its rank and $M^T$ its transpose. For any matrices $M_1$ and $M_2$ in $\mathbb{R}^{n \times m}$, $\langle M_1, M_2 \rangle_F$ denotes the canonical scalar product, {\it i.e.} $\langle M_1, M_2 \rangle_F = \text{Tr}(M_1^TM_2)$. For any $r \in [r_M]$, we denote $[M]_r$ the best $\rank r$ approximation of $M$ for the Frobenius norm.
In the model \eqref{eq: model matrix}, let us denote by $r^*$ the rank of $A^* X B^*$.

\section{Rank penalized learning} \label{sec: rank}

In this section we propose rank adaptive predictors and provide theoretical guarantees for their error. First we give explicit predictors under the assumption that the ranks of the parameter matrices are known, then a selection procedure will allow to provide a data-dependent rank selector and the associated rank-adaptive predictor. Even though we follow classical results for rank penalized (one-sided) matrix regression, e.g. \cite{bunea}, \cite{giraudEJS} and \cite{bingwegkamp}, we give details for the fixed rank two-sided matrix regression which is novel to the best of our knowledge. Surprisingly, explicit predictors can be proposed despite the identifiability issues of this model. Only after this, we proceed to rank selection and rank-adaptive learning.

\subsection{Prediction for given ranks} 
Let $r$ belong to $[n\wedge p \wedge r_X]$. 
Let us build explicit predictors $(\hat{A}_r, \hat{B}_r)$ solutions to the non-convex minimization problem 
\begin{equation} \label{opt: fixed rank minimization}
    \underset{\substack{A,B:\\ \rank A \wedge \rank B \leq r }}{\min} \| Y - A X B \|^2_F.
\end{equation}

Notice that the rank constraints on $A$ and $B$ use the same value $r$. Indeed the objective is to build a predictor for the signal $A^*XB^*$ which satisfies $\rank (A^*XB^*) \leq \min \left( r_{A^*}, r_X, r_{B^*} \right)$. In the steps of the proof of our results, we see that the upper bound of the risk depends on the ranks of $A^*$ and of $B^*$ only through their least value and no information can be recovered on the largest rank of the two. Hence it makes sense to look for $A$ and $B$ sharing the same rank as a dimension reduction technique without any impact on the final results.

The model \eqref{eq: model matrix} can be rewritten using the SVD of the observed matrix $Y$ and of the design matrix $X$ as
\begin{equation} \label{eq: diagonal model}
\Sigma_Y =     A_0^* \cdot \Sigma_X \cdot B_0^* + E_0,
\end{equation}
where $A_0^* = U_Y^T A^* U_X$,  $B_0^* = V_X^T B^* V_Y$ and $E_0:= U_Y^T \cdot E \cdot V_Y$. In the particular case where $E$ has independent entries with distribution $\mathcal{N}(0, \sigma^2)$ than so does $E_0$, see Lemma~\ref{lemma gaussian entries}. 
Now, $\Sigma_Y$ and $\Sigma_X$ are diagonal matrices, not necessarily squared, not necessarily full rank. Given the invariance of the Frobenius norm by left or right multiplication with orthogonal matrices, we get that for any matrices $A \in \mathbb{R}^{n \times m}$ and $B \in \mathbb{R}^{q \times p}$ we have
\begin{equation*} \label{eq: new param}
    \| Y - A X B \|^2_F  = \| \Sigma_Y - A_0 \Sigma_X B_0 \|^2_F,
\end{equation*}
where $A_0 = U_Y^T A U_X$ and $B_0=V_X^\top B V_Y$ are obtained via analogous transformations to those relating the true underlying parameters.

Obviously, matrices $A$ and $A_0$ have the same rank, and the same holds for $B$ and $B_0$. Therefore, solving \eqref{opt: fixed rank minimization} is equivalent to solving for $\hat{A}_{0r}$ and $\hat B_{0r}$ solutions of 
\begin{equation}\label{opt: diagonal form fixed rank}
    \underset{\substack{A_0,B_0: \\ \rank A_0 \wedge \rank B_0 \leq r }}{\min} \| \Sigma_Y - A_0 \Sigma_X B_0 \|^2_F.
\end{equation}

\begin{theorem} \label{th: fixed rank bounds diagonal}
    Let us define for $r \in [n \wedge p \wedge r_X]$    \begin{equation}\label{estimators_diag_fixed_rank}
        \hat{A_0}_{r} = Diag_{n,m}(\sigma_k(Y), \, 1\leq k \leq {r\wedge r_Y})\quad\text{and} \quad \hat{B_0}_{r} = Diag_{q,p}(\sigma_k(X)^{-1}, \, 1\leq k \leq {r}).
    \end{equation}
    Then, $(\hat{A_0}_{r}, \hat{B_0}_{r})$ belong to the set of solutions of problem \eqref{opt: diagonal form fixed rank} and the predictor $\hat{A_0}_r \Sigma_X \hat{B_0}_r$ satisfies for an absolute constant $C>0$ and for any $t>0$, the oracle inequality
    \begin{align*}
        \|A_0^* \Sigma_X B_0^* - \hat{A_0}_{r} \Sigma_X \hat{B_0}_{r}\|_F^2 &\leq 9 \underset{\substack{A_0,B_0:\\ \rank A_0 \wedge \rank B_0 \leq r }}{\inf}  \|A_0^* \Sigma_X B_0^* - A_0 \Sigma_X B_0\|_F^2 \\
        &+ 24 C \sigma^2 (1+t)^2 \cdot  r ({n}+ {p}),
    \end{align*}
with probability larger than $1-2\exp(-t^2(\sqrt{n}+ \sqrt{p})^2)$. 
    \end{theorem}

Next, from the explicit solutions $(\hat A_{0r}, \hat B_{0r})$ of \eqref{opt: diagonal form fixed rank} we deduce explicit solutions of \eqref{opt: fixed rank minimization}.

\begin{corollary} \label{cor: upper bound original problem}
Let us define for $r \in [n \wedge p \wedge r_X]$
\begin{equation}\label{estimators_fixed rank}
\hat{A}_r = U_Y \hat{A}_{0r} U_X^T \quad \text{and} \quad \hat{B}_r = V_X \hat{B}_{0r} V_Y^T,    
\end{equation} 
with $\hat A_{0r}$ and $\hat B_{0r}$ defined in \eqref{estimators_diag_fixed_rank}.
    Then $(\hat{A}_{r}, \hat{B}_{r})$ are solution to the problem \eqref{opt: fixed rank minimization} and the predictor $\hat A_r X \hat B_r$ satisfies for an absolute constant $C>0$ and for any $t>0$, the oracle inequality
    \begin{align*}
        \|A^* X B^* - \hat{A}_{r} X \hat{B}_{r}\|_F^2 &\leq 9 \underset{\substack{A,B:\\ \rank A \wedge \rank B \leq r }}{\inf}  \|A^* X B^* - A X B\|_F^2  + 24 C \sigma^2 (1+t)^2 \cdot  r ({n}+ {p}),
    \end{align*}
with probability larger than $1-2\exp({-t^2(\sqrt{n}+ \sqrt{p})^2})$. 
\end{corollary}
The proofs of Theorem~\ref{th: fixed rank bounds diagonal} and of Corollary~\ref{cor: upper bound original problem} can be found in Section~\ref{sec:proofs}. In the proofs we explicit the bias in terms of the unknown matrix parameters:
$$
\underset{\substack{A,B:\\ \rank A \wedge \rank B \leq r }}{\inf}  \|A^* X B^* - A X B\|_F^2 = \sum_{k=r+1}^{r^*} \sigma_k(A^* X B^*)^2 \cdot \mathbf{1}_{r < r^*}.
$$

\bigskip

Note that our choice for the couple of predictors $(\hat{A_0}_{r}, \hat{B_0}_{r})$ is not unique and we can easily derive families of solutions to the problem \eqref{opt: diagonal form fixed rank}. Each family of solutions can be turned into a solution to the problem \eqref{opt: fixed rank minimization}. Indeed, consider $ (\alpha \hat{A_0}_{r}, \dfrac{1}{\alpha} \hat{B_0}_{r})$ with arbitrary $\alpha >0$. Alternatively, let $\lambda_i $ for all $i \leq m \wedge q$ be arbitrary positive numbers, then 
$$(\hat{A_0}_r Diag_{m,m}(\lambda_1,\ldots, \lambda_{m \wedge q}), Diag_{q,q}(\lambda_1^{-1}, \ldots,\lambda_{m\wedge q}^{-1})\hat{B_0}_r) 
$$
give the same prediction. 
Let us see that the same transformations applied to the parameter matrices $A_0^*$ and $B_0^*$ also lead to the same signal matrix $A_0^* \Sigma_X B_0^*$. Indeed, the model is non-identifiable and so, without further strong assumptions, we can only hope to learn the global signal, and not the parameters of the model.

{\bf Alternative predictors.} Let us define a second couple of predictors $(\tilde A, \tilde B_r)$ producing exactly the same prediction as $(\hat A_r, \hat B_r)$ with the same theoretical properties, but having the advantage that $\tilde A$ is full rank and does not depend on $r$. Define
\begin{align*}
    \Tilde{A_0} = I_{n,m} & \quad  \text{and} \quad \Tilde{B_0}_r = Diag_{q,p}\left( \frac{\sigma_k(Y)}{\sigma_k(X)}, \, 1\leq k \leq {r\wedge r_Y}  \right)
\end{align*}
where $I_{n,m}$ denotes the identity matrix of dimension $n\times m$, whereas $\Tilde{B_0}_r$ has rank $r\wedge r_Y$. Using the analogous transformations we obtain
\begin{align*}
\Tilde{A} = U_Y I_{n,m} U_X^T & \quad \text{and} \quad \Tilde{B}_r = V_X \Tilde{B_0}_r V_Y^T.   
\end{align*} 
It is easy to see that Theorem~\ref{th: fixed rank bounds diagonal} is valid for $\Tilde{A_0}$ and $\Tilde{B_0}_r$, and that Corollary~\ref{cor: upper bound original problem} is valid for $\Tilde{A}$ and $\Tilde{B}_r$. 

\subsection{Rank-adaptive prediction}

In this section, we propose rank-adaptive predictors $(\hat A_{\hat r}, \hat B_{\hat r})$ which are selected from the family $\{(\hat A_r, \hat B_r): r \in [n \wedge p \wedge r_X]\}$ by a model selection procedure analogous to that of \cite{bunea}. 
Let us first define, for a generic matrix $M$ and any $\lambda >0$, the $\lambda-$rank of $M$ as
$$
r_M(\lambda) = 1 \vee \sum_{k=1}^{\rank M} \mathbf{1}_{\sigma_k(M)^2 \geq \lambda}.
$$

For given $ \lambda >0$, let 
\begin{equation} \label{opt: rank search minimization}
\hat r := \arg\min_{r \in [n \wedge p \wedge r_X]} \left\{ \|Y - \hat A_r X \hat B_r\|_F^2 + \lambda r \right\}.
\end{equation}
Consider the predictors introduced in \eqref{estimators_fixed rank} for the data-driven rank $\hat r$ as defined in \eqref{opt: rank search minimization}. The next Theorem extends the oracle inequality to the rank-adaptive predictors  $(\hat A_{\hat r}, \hat B_{\hat r})$ associated to the estimated rank $\hat r$ and to some $\lambda>0$ large enough.
\begin{theorem} \label{th: rank adaptive predictor upper bound}
The rank-adaptive predictors $(\hat A_{\hat r}, \hat B_{\hat r})$ associated to $\hat r$ in \eqref{opt: rank search minimization} and to $\lambda$ such that, for some absolute constant $C>0$ and for any $t>0$, $\lambda \geq  4 C (1+t)^2 \sigma^2 ({n}+ {p})$, satisfy the oracle inequality
    \begin{equation*}
        \|A^* X B^* - \hat A_{\hat r} X \hat B_{\hat r}\|_F^2 \leq \underset{r \in [n \wedge p \wedge r_X]}{\min} \left\{  9 \sum \limits_{k=r+1}^{r^*} \sigma_k(A^* X B^*)^2 \cdot \mathbf{1}_{r<r^*} + 6 \lambda r \right\},
    \end{equation*}
with probability larger than $1-2\exp({-t^2(\sqrt{n}+ \sqrt{p})^2})$.
\end{theorem}
Note that the minimum on the right-hand side of the previous display is always smaller than the value at $r=r^*$, giving under the assumptions of Theorem~\ref{th: rank adaptive predictor upper bound} that
$$
\|A^* X B^* - \hat A_{\hat r} X \hat B_{\hat r}\|_F^2 \leq  6r^*\lambda,
$$ 
with probability larger than $1-2\exp({-t^2(\sqrt{n}+ \sqrt{p})^2})$.\\
The bounds of order $r^*(n+p)$ attained by our procedure are analogous to those for the low-rank matrix regression models in \cite{rohde2011estimation} and \cite{giraudEJS}. Indeed, the 2MR model is more difficult than the MR model, ({\it i.e.} one of the matrices is known) and we will suppose known the matrix with larger rank in order to achieve the correct lower bounds. Thus the lower bounds for prediction in the low-rank MR model will be valid for our model.

\subsection{Consistent rank selection}\label{sec: rank selection}

We study the consistency of the rank selector $\hat r$ in \eqref{opt: rank search minimization} and see when it recovers the true rank $r^*$ with high probability. First, we show that, for properly chosen $\lambda$, the data-driven rank $\hat r$ is actually the unique solution and coincides with the $\lambda-$rank of $Y$, $\hat r = r_{Y}(\lambda)$. 

\begin{proposition}\label{prop: hat r 1}
If $\lambda > \sigma_{r_Y}(Y)^2$, there is a unique solution $\hat r$ to the optimisation problem in \eqref{opt: rank search minimization} and it is actually the $\lambda-$rank of $Y$, i.e. $\hat r = r_Y(\lambda)$.
\end{proposition}

Next, we prove that $\hat r$ recovers with high probability the $\lambda-$rank of $A^*XB^*$.
\begin{proposition} \label{prop: hat r 2}
    Let $\lambda >0$ and denote by $r^*(\lambda)$ the $\lambda-$rank of $A^* X B^*$ . If for some constant $c$ in (0,1),
    $\sigma_{r^*(\lambda)}(A^* X B^*)^2 > (1+c)^2 \lambda$ and $\sigma_{r^*(\lambda)+1}(A^* X B^*)^2 < (1-c)^2 \lambda$, then 
    $$
    \mathbb{P} (\hat r = r^*(\lambda)) \geq \mathbb{P} ( \|E\|_{op}^2 \leq c^2 \lambda).
    $$
    In particular, if $\lambda \geq 2C ({n} + {p}) \sigma^2 (1+t)^2/c^2$ for some absolute constant $C>0$ and for any $t>0$, then
    $\hat r = r^*(\lambda)$ with probability larger than $1-2\exp(-t^2 (\sqrt{n} + \sqrt{p})^2)$.
\end{proposition}

Finally, remember that the fact that $r^*(\lambda)$ coincides with the true underlying rank $r^*$ is equivalent to having $\sigma_{r^*}(A^* X B^*)^2 \geq \lambda >0$. The rank selector will then coincide with $r^*$ if $\lambda$ also satisfies $\sigma_1(E)^2 \leq c^2 \lambda$, for some absolute constant $c>0$. It is therefore necessary that a signal-to-noise ratio, given here by $\sigma_{r^*}(A^* X B^*)^2 / \sigma_1(E)^2$ be significant in order to have the true underlying rank selected by $\hat r$.
By combining this with the previous Propositions we get the following.
\begin{proposition}\label{prop: hat r 3}
    Let $\lambda >0$. If for some constant $c$ in (0,1), $\sigma_{r^*} (A^* X B^*)^2 > (1+c)^2 \lambda$, then 
    $$
    \mathbb{P}(\hat r = r^*) \geq  \mathbb{P} ( \|E\|_{op}^2 \leq c^2 \lambda).
    $$
    In particular, if $\lambda \geq 2 C ({n} + {p}) \sigma^2 (1+t)^2/c^2$ for some absolute constant $C>0$ and for any $t>0$, then
    $\hat r = r^*$ with probability larger than $1-2\exp(-t^2 (\sqrt{n} + \sqrt{p})^2)$.
\end{proposition}

\subsection{Data-driven rank-adaptive prediction}

The rank selector $\hat r$ in \eqref{opt: rank search minimization} is used for building consistent predictors as detailed in Theorem~\ref{th: rank adaptive predictor upper bound} provided that the condition $\lambda \geq  4 C (1+t)^2 \sigma^2 ({n}+ {p})$ is satisfied. However the noise parameter $\sigma$ is not known in general settings. Thus a data dependent rank selector is needed for building consistent predictors in those cases. Motivated by the previous case where $\sigma^2$ was supposed known, we proceed as follows. First, we change the penalty to $\lambda \cdot r \widehat{\sigma}_r^2 $ with  
$$
\widehat{\sigma}_r^2 = \frac 1{np} \|Y - \hat A_r X \hat B_r\|_F^2 .
$$
Note that in the particular case of Gaussian noise $\widehat{\sigma}_r^2$ estimates the variance $\sigma^2$ of the noise. 
Next, given a largest possible value for the true rank $r_{max} \leq n \wedge p \wedge r_X$, we define the data-driven rank selector
\begin{equation} \label{opt: rank search minimization unknown variance}
\bar r := \arg\min_{r \in [r_{max}]} \left\{ \|Y - \hat A_r X \hat B_r\|_F^2 + \lambda \cdot r \widehat{\sigma}_r^2  \right\}.
\end{equation}
Finally, we use the predictors $(\hat A_{\bar r}, \hat B_{\bar r})$. 
The next theorem extends the upper bounds of Theorem~\ref{th: rank adaptive predictor upper bound} to these data-driven rank-adaptive predictors. 

\begin{theorem} \label{th: rank adaptive predictor upper bound unknown variance}
The data-driven rank-adaptive predictors $(\hat A_{\bar r}, \hat B_{\bar r})$ associated to $\bar r$ in \eqref{opt: rank search minimization unknown variance} with $r_{max} \leq n \wedge p \wedge r_X$, and to $\lambda = (1+\varepsilon) np /(r_{max} \vee r_Y)$ for some $\varepsilon >0$, satisfy for some absolute constant $C>0$ and for any $t>0$ the oracle inequality
    \begin{align*}
        \|A^* X B^* - \hat A_{\bar r} X \hat B_{\bar r}\|_F^2 \leq \underset{r \in [r_{max}]}{\min} &\left\{ 9 \|A^*X B^* - \hat A_r X \hat B_r\|_F^2 + 6(1+\varepsilon) \cdot r  \sigma_{r+1}(A^* X B^*)^2 \right\} \\
        & + 12C(2+\varepsilon)(1 + t)^2 \cdot  \sigma^2 r_{max} (n + p) ,
    \end{align*}
with probability larger than $1-2\exp({-t^2(\sqrt{n} + \sqrt{p})^2})$.
\end{theorem}
Apply the Corollary \ref{cor: upper bound original problem}, to get under the assumptions of Theorem \ref{th: rank adaptive predictor upper bound unknown variance} that
    \begin{align*}        
    \|A^* X B^* - \hat A_{\bar r} X \hat B_{\bar r}\|_F^2 \leq \underset{r \in [r_{max}]}{\min} &\left\{ 9^2 \underset{\substack{A,B:\\ r_A \wedge r_B \leq r}}{\inf} \|A^*X B^* - A_r X B_r\|_F^2 + 6(1+\varepsilon) \cdot r  \sigma_{r+1}(A^* X B^*)^2 \right\} \\
        & + 12(20+\varepsilon)C (1 + t)^2 \cdot  \sigma^2 r_{max} (n + p) ,
    \end{align*}
with probability larger than $1-2\exp({-t^2(\sqrt{n} + \sqrt{p})^2})$.

Note that the minimum on the right-hand side of the previous display is always smaller than its value at $r=r^*$ if $r_{max}$ is larger than $r^*$, giving under the assumptions of Theorem~\ref{th: rank adaptive predictor upper bound unknown variance} that
$$
\|A^* X B^* - \hat A_{\bar r} X \hat B_{\bar r}\|_F^2 \leq 12(20+\varepsilon) C(1+t)^2\cdot \sigma^2 r_{max}(n+p) .
$$
In order to compare to the previous results, note that the upper bound derived from Theorem~\ref{th: rank adaptive predictor upper bound} for the value $r=r^*$ and the least value $\lambda = 4C(1+t)^2\sigma^2(n+p)$ gives the very similar bound
$$
\|A^* X B^* - \hat A_{\hat r} X \hat B_{\hat r}\|_F^2 \leq  24C(1+t)^2 \cdot \sigma^2 r^*(n+p).
$$

From a computational point of view, it is preferable to change $\widehat{\sigma}_r^2$ in some cases. For example, we use in our numerical simulations
$$
\widehat{\sigma}_r^2 = \frac 1{np-(m \wedge q) r_X} \|Y - \hat A_r X \hat B_r\|_F^2 
$$ 
when $n\geq m$, $p\geq q$ and thus $np>(m \wedge q) r_X$. It is straightforward to prove the analogue of Theorem~\ref{th: rank adaptive predictor upper bound unknown variance} by considering $\lambda = (1+\varepsilon)(np-(m \wedge q) r_X)/(r_{max}\vee r_Y)$.

\section{Nuclear norm penalized learning} \label{sec: nuclear}

Nuclear norm penalized least squares is known to exhibit good properties, see \cite{bach2008consistency} or \cite{negahban2011estimation}. Hence it may show advantages over rank-penalized methods.  
Let us define the nuclear norm penalized (NNP) optimisation problem 
\begin{equation} \label{opt: nuclear pen}
    \underset{A, B}{\min}\,  \|Y - AXB\|_F^2 + 2 \lambda \cdot \|AXB\|_*,
\end{equation}
for some $\lambda >0$. The objective of the optimization problem is non-jointly convex in $A$ and $B$. Note that in matrix regression (when $A^*$ is the identity matrix) the nuclear norm of $XB$ has been used , see \cite{koltchinskii}, or other adaptive forms depending on the feature matrix $X$, \cite{ChenDongChan}. However, we exhibit explicit predictors belonging to the set of solutions of this problem and show an oracle inequality they satisfy.

\begin{theorem} \label{th: nuc norm predictor upper bound}
The predictors $(\bar A, \bar B)$ defined by
\begin{equation}\label{eq: nuclear pen}
    \bar A = U_Y I_{n,m} U_X^\top \quad \text{and} \quad
    \bar B = V_X \cdot Diag_{q,p} \left( \frac{(\sigma_k(Y) - \lambda)_+}{\sigma_k(X)}, \, 1\leq k \leq r_Y \wedge r_X\right) V_Y^\top
\end{equation}
are solutions to the problem in \eqref{opt: nuclear pen}. Moreover, if $\lambda$ is such that, for some absolute constant $C>0$ and for any $t>0$, $\lambda \geq  2 C (1+t)^2 \sigma^2 ({n}+ {p})$, they satisfy the oracle inequality
    \begin{equation*}
        \|A^* X B^* - \bar A X \bar B\|_F^2 \leq 9 \underset{r \in [n \wedge p \wedge r_X]}{\min} \left\{  \sum \limits_{k=r+1}^{r^*} \sigma_k(A^* X B^*)^2 \cdot \mathbf{1}_{r<r^*} + 16 \lambda r \right\},
    \end{equation*}
with probability larger than $1-2\exp({-t^2(\sqrt{n}+ \sqrt{p})^2})$.
\end{theorem}
The proof can be found in Section~\ref{sec:proofs}.

{\bf Remark. }Another approach could be to consider the model under the vectorized form \eqref{eq: vec form model} and solve the problem
\begin{equation*}
    \min_{A,B} \|vec(Y)^\top - vec(X)^\top \cdot A^\top \otimes B\|_2^2 + 2 \lambda \|A^\top \otimes B\|_*,
\end{equation*}
for some $\lambda >0$. Recall that $A^\top \otimes B$ denotes the tensor product of matrices $A^\top$ and $B$ and that we can write
$\|A^\top \otimes B\|_* = \sum_{k, j \geq 1} \sigma_k(A) \sigma_j(B).$
However, the features are 1-dimensional and we loose the structured information contained in the original matrix $X$. This approach could make more sense in the case of repeated observation $(Y_t, X_t)$ for $t$ in $[T]$, by stacking the rows $vec(Y_t)^\top$ and $vec(X_t^\top)$ into matrices $\mathbb Y$ and $\mathbb X$, respectively, and do a classical matrix regression. Even so, the usual assumptions on the feature matrix $\mathbb X$ in order to achieve good prediction are not reasonable in this context as they are not much related to the original matrix data sets $X_t$, $t$ in $[T]$.

{\bf Remark (Sufficient conditions for identifiability)} 
We have indicated at several times that many couples of matrices $(A,B)$ solve the equation $M=AXB$ for a given matrix $M$. Given the SVD of the matrix $M$, we may reduce the dimensionality of the problem by choosing the solution $( A,  B)$ given by $ A = U_M A_0 U_X^\top$ and $ B = V_X  B_0 V_M^\top$, with $A_0$ and $B_0$ diagonal matrices such that 
$$
\sigma_k(A) \sigma_k(X) \sigma_k(B) = \sigma_k(M), \quad 
\text{for all } k \leq r_X \wedge r_A \wedge r_B. 
$$
Thus, even under diagonal forms we can only identify the product of respective singular values of $A$ and $B$. We can only hope to identify matrices $A$ and $B$ under very restrictive conditions where $X^\top X$ has full rank and either the matrix $A$ or the matrix $B$ is assumed to have known singular values, {\it e.g.} like a projector with singular values 1 or 0. Few other setups are known to be identifiable in the literature of factorisation of matrices, {\it e.g.} non-negative matrix factorisation (NMF), see \cite{donoho2003does}, NMF for topic models \cite{ke2022using}, \cite{bing2020optimal}, \cite{klopp2021assigning} or covariance matrix factorization \cite{fan2011high}. 

\section{Numerical Results} \label{sec: numerics}
Let us set the dimensions of the observed matrix $Y$ to be $n=100$ and $p = 300$, the dimensions of the design matrix $X$ to be $m = 50$ and $q = 60$. We randomly generate three matrices: $A^*$, $B^*$, and $X$, with independent random gaussian entries with mean $0$ and variance $1$. These matrices are then projected onto the best low-rank matrix approximation, with the matrix $A^*$ having a rank $r_A^* = 16$, the matrix $B^*$ having a rank $r_B^* = 12$, and the matrix $X$ having a rank $r_X = 25$. The signal matrix is defined as $A^* X B^*$ and shows a rank of $12$ in all experiments. We also define various settings for the variance $\sigma^2$ of the Gaussian noise $E$ so that the signal-to-noise ratio $SNR:={\sigma_{r^*}(A^*XB^*)^2}/{\sigma_1(E)^2}$ varies approximately in the range $[0.5, 2]$. 

\begin{figure}
\minipage{0.99\textwidth}
 \centering
  \includegraphics[width=\linewidth]{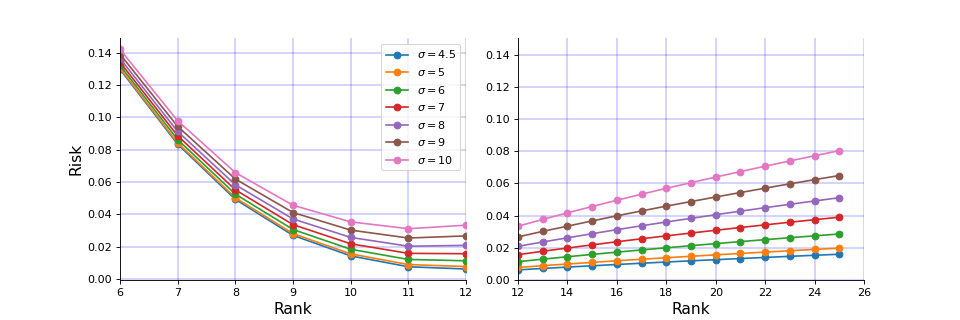}
\endminipage\hfill
\caption{Evolution of the risk $\dfrac{\|\hat{A}_r X \hat{B}_r - A^* X B^*\|_F^2}{\|A^* X B^*\|_F^2}$ in function of $r$ for different values of $\sigma$}
\label{fig:risk_evol_rank}
\end{figure}

Figure~\ref{fig:risk_evol_rank} illustrates the prediction performances of the predictor $\hat{A}_r X \hat{B}_r$, defined in \eqref{estimators_fixed rank}, for different values of $r$. For $\sigma < 8$ giving the $SNR$ approximately above the value 1, the prediction risk decreases when the rank increases while remaining bounded from above by $12$ and then increases with the rank when the rank is above $12$. For $\sigma \geq 8$ giving the $SNR$ below the value 1, the prediction risk decreases when the rank increases while remaining bounded from above by $11$ and then increases with the rank when the rank is above $11$. It highlights that the best predictor is achieved when $r = r^* = 12$ for small noise variance levels (\textit{i.e.} $\sigma < 8$) and when $r = 11$ for strong noise variance levels (\textit{i.e.} $\sigma \geq 8$). This shows that there is a strong overfitting phenomenon in the case of strong noise and that it is therefore better to slightly underestimate the rank in these situations. 

\begin{figure}
\minipage{0.99\textwidth}
 \centering
  \includegraphics[width=\linewidth]{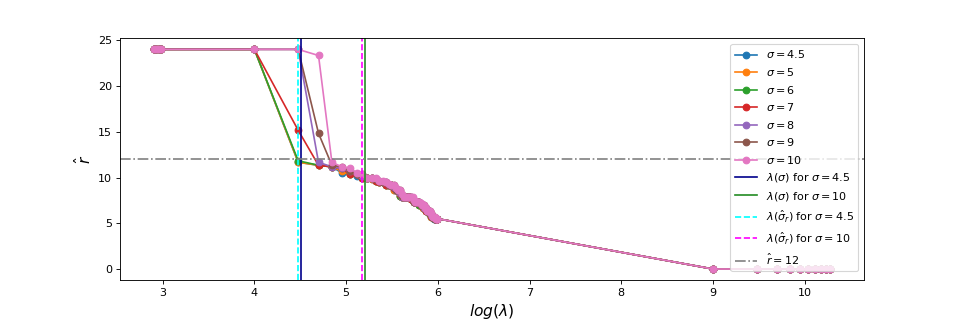}
\endminipage\hfill
\caption{Evolution of the estimated $\hat{r}$ as a function of $\log(\lambda)$ for different values of $\sigma$}
\label{fig:estimated_rank_lambda}
\end{figure}

Figure~\ref{fig:estimated_rank_lambda} represents the predicted $\hat{r}$, defined in \eqref{opt: rank search minimization}, for various values of $\lambda$. Independently of the noise variance level, for small values of $\lambda$ the estimated $\hat{r}$ is maximal and there is $\hat{r} = r_X = 25$. This illustrates the previously exposed overfitting phenomenon, that is the higher the rank $r$, the lower the error $\|Y - \hat{A}_r X \hat{B}_r\|_F^2$.  As $\lambda$ increases the penalty on the rank $r$ becomes more important in the minimization procedure and $\hat{r}$ decreases. However, for moderate values of $\lambda$ (\textit{i.e.} approximately $\log(\lambda) \leq 5$) the smaller the noise variance level $\sigma$, the faster $\hat{r}$ decreases. Ultimately, for large values of $\lambda$ (\textit{i.e.} approximately $\log(\lambda) > 5$) the rate of decay of $\hat{r}$ as a function of $\lambda$ no longer depends on $\sigma$. 

The numerical value of $\lambda$ is an important issue. We exhibit explicit (fast to calculate) procedures for the choice of this tuning parameter. In the case of {\it known noise variance}, the rule of thumb suggested by \cite{birge2007minimal} indicates to choose  
$$
\lambda(\sigma) = 2 C ({n} + {p}) \sigma^2 (1+t)^2
$$ in Theorem~\ref{th: rank adaptive predictor upper bound} with $t=0$, and $C=2$. The two solid vertical lines represent $\lambda(4.5)$ (blue) and $\lambda(10)$ (green). With these choices of the tuning parameter we get successful estimators of the underlying rank of the signal $\hat{r} \approx 12 = r^*$. We underline that in the small noise regime the rank is slightly overestimated and in the strong noise regime it is slightly underestimated. This behaviour perfectly matches the results drawn from Figure~\ref{fig:risk_evol_rank} showing that overestimating the rank in small noise regime does not impact the performances and slightly underestimating it in strong noise regime improves the performances. 

However, in real world applications the noise has {\it unknown variance}. This raises the question of how to choose a data-driven $\lambda$ in this case, without deteriorating the prediction. This situation is more challenging as it first requires an estimator of $\sigma^2$ before using the previously exposed rule of thumb. We choose the initial value of $r$ equal to $r_X \wedge n \wedge p$ and propose the $r$-dependent estimator $\widehat{\sigma}^2_r := \dfrac{\|Y - \hat{A}_r X \hat{B}_r \|_F^2}{np - (m \wedge q) r_X}$. 
It allows to compute the previously defined $\lambda(\widehat{\sigma}_r)$ and using this data-driven tuning parameter we produce the rank estimator $\bar{r}$. This procedure takes $r$ as an argument and returns $\lambda(\widehat{\sigma}_r)$ and $\bar{r}$. However, when $r$ is substantially larger than $r^*$, $\hat{A}_r X \hat{B}_r$ is overfitting $Y$ and performing this procedure once will not lead to a satisfying output $\bar{r}$. Hence we iterate while $\bar{r}< r$. We note $\lambda(\widehat{\sigma}_{\bar r})$ and $\bar r$ the final outputs of the procedure. The two dashed vertical lines represent $\lambda(\widehat{\sigma}_{\bar r})$ when $\sigma = 4.5$ (cyan) and $\sigma = 10$ (magenta). The proposed procedure exhibits great numerical properties.

Finally, numerical simulations generated in the same context, with different values for the true underlying ranks, show similar excellent prediction bounds, combined with correct rank selection. Together with the current case where $\min(r_A^*, r_X, r_B^*) = r_B^*$, we have explored successfully the cases $\min(r_A^*, r_X, r_B^*) = r_A^*$, $\min(r_A^*, r_X, r_B^*) = r_X$ and $\min(r_A^*, r_X, r_B^*) = r_A^* = r_X = r_B^*$.

\section{Proofs} \label{sec:proofs}

{\bf Basic facts} For any matrix $M \in \mathbb{R}^{n \times m}$, $\|M\|_*^2 \leq r_M \|M\|_F^2$. In addition, for any matrices $M_1$ and $M_2$ in $\mathbb{R}^{n \times m}$, the following inequalities hold $\langle M_1, M_2 \rangle_F \leq \|M_1\|_* \|M_2\|_{op}$ and $\|M_1 + M_2\|_F \leq \|M_1\|_F + \|M_2\|_F$. Furthermore, if we set $a = \rank M_1 \wedge \rank M_2$ then $\langle M_1, M_2 \rangle_F \leq \|M_1\|_{(2, a)} \|M_2\|_{(2, a)}$.

\begin{lemma} \label{lemma gaussian entries} Let $E$ be a $n\times p$ random matrix whose entries are independent and having Gaussian distribution $\mathcal{N}(0, \sigma^2)$. If $U$ and $V$ belong to $\mathcal{O}_n$ and $\mathcal{O}_p$ respectively, then $E_0:= U^\top E V$ has independent entries with Gaussian distribution $\mathcal{N}(0, \sigma^2)$.
\end{lemma}
\begin{proof}[Proof of Lemma~\ref{lemma gaussian entries}] Note that we can vectorize the matrix $E_0$ and get that
$$
vec(E_0) = (V^\top \otimes U^\top) \cdot vec(E),
$$
where $vec(E)$ is a Gaussian vector of dimension $np$, centered, with variance $\sigma^2 I_{np}$. Moreover, the tensor product $V^\top \otimes U^\top$ belongs to $\mathcal{O}_{np}$, thus $vec(E_0)$ is still a Gaussian vector with distribution $\mathcal{N}_{np}(0, \sigma^2 I_{np})$.
\end{proof}

Recall that, for an arbitrary matrix $M$, we denote $U_M \Sigma_M V_M^\top$ its SVD. 
\begin{lemma}\label{lemma:projection r}
If $M^*$ is a $n \times p$ matrix of rank $r^*$, than for any $r \leq n \wedge p$, we have
$$
\inf_{M:\rank M \leq r}  \|M-M^*\|_F^2 
 = \sum_{k=r+1}^{r^*} \sigma_k(M^*)^2 \cdot \mathbf{1}_{r < r^*},
$$
and the infimum is attained by the projection $[ M^*]_r$ of $M^*$ on the space of $n \times p$ matrices with rank $r$ given by the matrix
$$[ M^*]_r = U_{M^*} \cdot Diag_{n,p}( \sigma_1(M^*), ..., \sigma_{r \wedge r^*}(M^*)) \cdot V_{M^*}^\top.$$
\end{lemma}

\subsection{Proof of Theorem \ref{th: fixed rank bounds diagonal}}

Let $r \in [n \wedge p \wedge r_X]$ and $(\hat{A_0}_{r}, \hat{B_0}_{r})$ defined in \eqref{estimators_diag_fixed_rank}. Let us denote here $M_0^* = A_0^* \Sigma_X B_0^*$ and $\hat M_0 = \hat{A_0}_{r} \Sigma_X \hat{B_0}_{r} $. By construction, $\hat M_0$ is the projection $[\Sigma_Y]_r$ of $\Sigma_Y$ onto the set of matrices with rank less than or equal to $r$, in the sense of Lemma~\ref{lemma:projection r}. Therefore,
\begin{align*}
    \|\Sigma_Y - \hat{M}_0 \|_F^2 &\leq \|\Sigma_Y - [M_0^*]_r\|_F^2 
\end{align*}
We recall that in our model $\Sigma_Y = M_0^* +E_0$ which leads to
\begin{align*}
    \|M_0^* - \hat M_0 + E_0\|_F^2 & \leq \|M_0^* - [M_0^*]_r +E_0\|_F^2.
\end{align*}
We expand the squares and arrange terms to get
\begin{align*}
    \|M_0^* - \hat M_0\|_F^2   
    &\leq \|M_0^* - [M_0^*]_r\|_F^2 + 2 \langle \hat M_0 - [M_0^*]_r, E_0 \rangle_F .
\end{align*}
Now, since $\rank(\hat M_0) = r$ and $\rank([M_0^*]_r) \leq r$, we get that $\rank(\hat M_0 - [M_0^*]_r) \leq 2r$. This inequality gives
\begin{align*}
    \|M_0^* - \hat M_0 \|_F^2 &\leq \|M_0^* - [M_0^*]_r \|_F^2 + 2 \|E_0\|_{(2, 2r)} \cdot \| \hat M_0 - [M_0^*]_r\|_{(2, 2r)} \\
    &\leq \|M_0^* - [M_0^*]_r \|_F^2 + 2 \|E_0\|_{(2, 2r)} \cdot \| \hat M_0 - [M_0^*]_r\|_F\\
    & \leq \|M_0^* - [M_0^*]_r \|_F^2 + 2 \|E_0\|_{(2, 2r)} \cdot\left( \| \hat M_0 - M_0^*\|_F +  \| M_0^* - [M_0^*]_r\|_F \right)    .
\end{align*}
We apply the inequality $2xy \leq \alpha x^2 + \alpha^{-1} y^2$ with $x, y \geq 0$ and $\alpha >0$. We obtain, for real numbers $\alpha>1$ and $\beta>0$, 
\begin{align*}
   (1-\alpha^{-1} ) \cdot \|M_0^* - \hat M_0 \|_F^2 &\leq (1+\beta^{-1}) \cdot \|M_0^* - [M_0^*]_r \|_F^2 + (\alpha + \beta) \cdot \|E_0\|_{(2, 2r)}^2.
\end{align*}
Let us use that $\|E_0\|_{(2, 2r)}^2 \leq 2r \cdot \|E_0\|_{op}^2$ and Lemma~\ref{lemma:projection r} to further get
\begin{align} \label{eq: upper bound v1}
    \|M_0^* - \hat M_0 \|_F^2 &\leq \dfrac{1+\beta^{-1}}{1-\alpha^{-1}} \cdot \inf_{M:\rank M \leq r}\|M_0^* - M\|_F^2 + \dfrac{\alpha + \beta}{1-\alpha^{-1}}\cdot 2r \|E_0\|_{op}^2.
\end{align}
Noticing that for any matrices $ A_0, B_0$ having rank less than or equal to $r$, $\rank (A_0 \Sigma_X B_0) \leq r_{A_0} \wedge r_X \wedge r_{B_0} \leq r$, we deduce that 
$$
\inf_{M:\rank M \leq r} \|M_0^* - M\|_F^2 \leq  \underset{\substack{A_0,B_0:\\ \rank A_0 \wedge \rank B_0 \leq r }}{\inf} \|M_0^* - A_0 \Sigma_X B_0\|_F^2.
$$
Indeed, the second inf is taken over a possibly smaller family of matrices. We actually show that equality holds in the previous display. Indeed, by Lemma~\ref{lemma:projection r} we have that $\inf_{M:\rank M \leq r} \|M_0^* - M\|_F^2 = \sum_{k=r+1}^{r^*} \sigma_k(M_0^*)^2\cdot \mathbf{1}_{r < r^*}$, where $r^*= \rank (M_0^*)$. Recall that $M_0^*=A_0^* \Sigma_X B_0^*$ is a product of diagonal matrices, giving that $r^* = \min(r_X, r_{A_0^*}, r_{B_0^*}) $ and $\sigma_k(M_0^*) = \sigma_k(A_0^*) \sigma_k(X) \sigma_k(B_0^*) \cdot \mathbf{1}_{k \leq r^*}$. 
Thus, the particular choice $$
A_{0r}:=Diag_{n,m}(\sigma_1(A_0^*),\ldots , \sigma_{r\wedge r_{A_0^*}}(A_0^*)) \text{ and }
B_{0r}:=Diag_{q,p}(\sigma_1(B_0^*),\ldots , \sigma_{r\wedge r_{B_0^*}}(B_0^*))
$$ solves exactly the problem giving $M_0^* = A_{0r} \Sigma_X B_{0r}$. Finally,
\begin{align}\label{eq: equality inf}
    \inf_{M:\rank M \leq r} \|M_0^* - M\|_F^2 =\underset{\substack{A_0,B_0: \\\rank A_0 \wedge \rank B_0 \leq r }}{\inf} \|M_0^* - A_0 \Sigma_X B_0\|_F^2.
\end{align}

Plugging this into \eqref{eq: upper bound v1} and considering the particular choice $\alpha = 3/2$ and $\beta = 1/2$ give the theorem:
\begin{align*} 
\|A_0^* \Sigma_X B_0^* - \hat{A_0}_{r} \Sigma_X \hat{B_0}_{r}\|_F^2 \leq 9 \underset{\substack{A_0,B_0:\\ \rank A_0 \wedge \rank B_0 \leq r }}{\inf} \left( \|A_0^* \Sigma_X B_0^* - A_0 \Sigma_X B_0\|_F^2 \right) + 12 r \|E_0\|_{op}^2 .
\end{align*}

The last step is the high-probability bound on $\|E_0\|_{op}$. Recall that $E_0 = U_Y^\top E V_Y$ with $U_Y$ in $\mathcal{O}_n$ and $V_Y$ in $\mathcal{O}_p$ and therefore $E_0$ and $E$ have the same singular values. Therefore $\|E\|_{op}= \|E_0\|_{op}$. 
The noise matrix $E$ has independent, centered, $\sigma-$sub-Gaussian entries and its spectral norm verifies (see \cite{vershynin}) for some absolute constant $C>0$
\begin{equation} \label{ineq: concentration}
    \mathbb{P} \left( \|E\|_{op}^2 \leq 2C \sigma^2 \cdot (1+t)^2 ({n}+ {p}) \right) \geq 1 - 2 e^{-t^2(\sqrt{n}+ \sqrt{p})^2}, \quad \text{for any } t>0.
\end{equation}
Moreover, $ \E \left[ \|E\|_{op} \right] \leq \sqrt{C} \sigma (\sqrt{n} + \sqrt{p}).$

\subsection{Proof of Corollary~\ref{cor: upper bound original problem}}

Recall the notation $M^*_0 = A_0^* \Sigma_X B_0^*$ and $\hat M_0 = \hat{A_0}_{r} \Sigma_X \hat{B_0}_{r}$ with $\hat{A_0}_r$ and $\hat{B_0}_r$ given by \eqref{estimators_diag_fixed_rank} and let us denote $M^* = A^* X B^*$ and $\hat M = \hat{A}_{r} X \hat{B}_{r}$ with $\hat A_r$ and $\hat B_r$ given by \eqref{estimators_fixed rank}. Notice that the Frobenius norm and the rank are invariant under left or right multiplication by orthogonal matrices. Therefore, we follow the lines of the proof of Theorem~\ref{th: fixed rank bounds diagonal} and see that $\| Y - \hat{M} \|^2_F = \| \Sigma_Y - \hat{M}_{0} \|^2_F$ and $\rank M^* = \rank M_0^* = r^*$. Also, $\hat M$ is the projection $[Y]_r$ of $Y$ on the space of matrices with rank less than or equal to $r$. 
Finally, the equality \eqref{eq: equality inf} can be pushed forward 
$$
    \inf_{M:\rank M \leq r} \|M_0^* - M\|_F^2
    =\underset{\substack{A_0,B_0: \\\rank A_0 \wedge \rank B_0 \leq r }}{\inf} \|M_0^* - A_0 \Sigma_X B_0\|_F^2
    =\underset{\substack{A,B: \\\rank A \wedge \rank B \leq r }}{\inf} \|M^* - A X B\|_F^2.
$$
Indeed, we have one-to-one transformations of $A_0, \, B_0$ into $A,\, B$, respectively, and equality of the Frobenius norms. This finishes the proof.

\subsection{Proof of Theorem~\ref{th: rank adaptive predictor upper bound}}

By definition of $\hat r = \hat r(\lambda)$, we have that, for all $r \in [n \wedge p \wedge r_X]$,
\begin{align*}
    \|Y - \hat A_{\hat r} X \hat B_{\hat r}\|_F^2 + \lambda \hat r &\leq \|Y - \hat A_r X \hat B_r\|_F^2 + \lambda r. 
\end{align*}
Since $\hat A_r X \hat B_r$ is the projection $[Y]_r$ of $Y$ on the space of matrices $M$ with $\rank M\leq r$, we get that for all matrices $A$ and $B$ such that $\rank A \wedge \rank B \leq r$
$$
\|Y - \hat A_r X \hat B_r\|_F^2 \leq \|Y - A X B\|_F^2.
$$
Indeed, $\rank(AXB) \leq r$ and Pythagora's theorem gives the former inequality. We deduce that
\begin{align*}
    \|Y - \hat A_{\hat r} X \hat B_{\hat r}\|_F^2 + \lambda \hat r &\leq \|Y - A X B\|_F^2 + \lambda r. 
\end{align*}
Next, replace $Y=A^* X B^* + E$, expand the squares and rearrange terms to get
\begin{align*}
    \| A^* X B^* - \hat A_{\hat r} X \hat B_{\hat r}\|_F^2  &\leq \|A^* X B^* - A X B \|_F^2 + \lambda (r - \hat r ) \\
    & + 2 \langle E, \hat A_{\hat r} X \hat B_{\hat r} -  A X B \rangle.
\end{align*}
Let us denote by $\hat M(\hat r)  = \hat A_{\hat r} X \hat B_{\hat r}$, $M(r) = A X B$ and see that $\rank(\hat M(\hat r) -  M(r)) \leq \hat r +r$. We have
\begin{align*}
    \langle E , \hat A_{\hat r} X \hat B_{\hat r} - A X B \rangle & \leq \|E \|_{op} \cdot \| \hat M (\hat r) - M(r)\|_* \\
    &\leq \|E \|_{op} \cdot \sqrt{\hat r +r} \| \hat M (\hat r) -  M(r)\|_F \\
    &\leq \|E \|_{op} \cdot \sqrt{\hat r +r} (\| M^*- \hat M (\hat r)\|_F + \|M^* -  M(r) \|_F).
\end{align*}
Then, using twice the inequality $2xy \leq \alpha x^2 + \alpha^{-1}y^2$ with $x, y \geq 0$ and $\alpha > 0$, we obtain for arbitrary real numbers $\alpha>1$, $\beta>0$:
\begin{align*}
    (1 - \alpha^{-1}) \|M^* - \hat M(\hat r) \|_F^2 &\leq (1 + \beta^{-1}) \|M^* -  M(r)\|_F^2 \\
    & + (\alpha +\beta) \|E\|_{op}^2 (r + \hat r) + \lambda(r- \hat r).
\end{align*}
Consequently, if $(\alpha +\beta) \|E\|_{op}^2 \leq \lambda$:
$$
(1 - \alpha^{-1}) \|M^* - \hat M(\hat r) \|_F^2 \leq (1 + \beta^{-1}) \|M^* -  M(r) \|_F^2 + 2 \lambda r,
$$
for all $r$ in $[n \wedge p \wedge r_X]$ and all $M(r) = A X B$ with $\rank A\wedge \rank B \leq r$. We get the result by replacing again $\alpha=3/2$ and $\beta=1/2$. Then we use that 
\begin{align*}
    \underset{\substack{A, B \\ \rank A\wedge \rank B \leq r}}{\min} \| A^* X B^* - A X B \|_F^2 = \sum \limits_{k = r + 1}^{r^*} \sigma_k(A^* X B^*)^2
\end{align*}
and the high-probability bounds in \eqref{ineq: concentration}.

\subsection{Proofs of results in Section~\ref{sec: rank selection}}

\begin{proof}[Proof of Proposition~\ref{prop: hat r 1}]
    For any $r$ in $[n \wedge p \wedge r_X]$, we have that $\hat A_r X \hat B_r = [Y]_r$ is the projection of $Y$ on the space of matrices having rank smaller than or equal to $r$. 
    Now, write
    \begin{align*}
        F(r): & = \|Y - \hat A_r X \hat B_r\|_F^2 + \lambda r\\
       & = \sum_{k=r+1}^{r_Y} \sigma_k(Y)^2 \cdot \mathbf{1}_{r < r_Y} + \lambda r\\
       & = \sum_{k=r+1}^{r_Y} (\sigma_k(Y)^2 - \lambda) \cdot \mathbf{1}_{r < r_Y} + \lambda r_Y.
    \end{align*}
It is easy to see that $F$ as a function of $r$ has a unique minimum at $r_Y(\lambda)$ if $\lambda > \sigma_{r_Y}(Y)^2$, but is minimal and constant for $r = r_Y,\ldots, (n \wedge p\wedge r_X)$ whenever $\lambda \leq \sigma_{r_Y}(Y)^2$.
\end{proof}

\begin{proof}[Proof of Proposition~\ref{prop: hat r 2}]
By definition of $\hat r$, we have $k > \hat r$ if and only if $\lambda > \sigma_k(Y)^2$ and $k < \hat r$ if and only if $\lambda \leq \sigma_{k+1}(Y)^2$.     In our model $Y = A^* X B^* + E$, the Weyl inequality gives $| \sigma_k(A^*XB^*) - \sigma_k(Y)| \leq \sigma_1(E)$ for all $k$. The events on $\hat r$ can be written in terms of $\sigma_1(E) = \|E\|_{op}$ as follows. We have
\begin{align*}
    \{k > \hat r\} \quad &\text{implies} \quad \lambda > (\sigma_k(A^*XB^*) - \sigma_1(E))^2 , \\
    \{k < \hat r\} \quad &\text{implies} \quad \lambda \leq (\sigma_{k+1}(A^* X B^*) + \sigma_1(E))^2 .
\end{align*}
Thus $\{\hat r \not = k\}$ implies either $\sigma_1(E) > \sigma_k(A^* X B^*) - \sqrt{\lambda}$ or $\sigma_1(E) \geq \sqrt{\lambda} - \sigma_{k+1}(A^* X B^*)$. 
Let us take $k = r^*(\lambda)$. Then the assumption  that $\sigma_{r^*(\lambda)}(A^*XB^*) > (1+c)\sqrt{\lambda}$ gives that  $\sigma_1(E) > c \sqrt{\lambda}$ and the assumption that $\sigma_{r^*(\lambda)+1}(A^*XB^*) < (1-c)\sqrt{\lambda}$ gives also that $\sigma_1(E) > c \sqrt{\lambda}$. Thus,
\begin{align*}
    \P \left(\hat r \neq r^*(\lambda) \right) &\leq \P \left( \sigma_1(E) > c \sqrt{\lambda} \right).
\end{align*}
The proof is finished using the inequality \eqref{ineq: concentration}.
\end{proof}

\subsection{Proof of Theorem~\ref{th: rank adaptive predictor upper bound unknown variance} }

The optimization problem \eqref{opt: rank search minimization unknown variance} can be written, after replacing $ \widehat \sigma_r^2$, as follows
$$ \bar r  \in \arg\min_{r \in [r_{max}]} \|Y - \hat A_r X \hat B_r\|_F^2 \left( 1 + \frac{\lambda r}{np} \right) . $$
We denote by $\bar M = \hat A_{\bar r} X \hat B_{\bar r}$, $\hat M_r = \hat A_{r} X \hat B_{r}$ and $M^* = A^* X B^*$.
With this notation it follows that, for $r \leq r_{max}$, 
$$ 
\| Y - \bar M \|_F^2 \left( 1 + \frac{\lambda \bar r}{np} \right) \leq \| Y - \hat{M}_r \|_F^2 \left( 1 + \frac{\lambda  r}{np} \right).
$$
Developing the squares and using the equality $Y = M^* + E$, we get
$$
\| M^* - \bar M \|_F^2 \leq \| M^* - \hat{M}_r \|_F^2 + 2 \langle E, \bar M - \hat M_r \rangle_F + \frac{\lambda  r}{np} \| Y - \hat M_r \|^2_F - \frac{\lambda \bar r}{np} \| Y - \bar M \|_F^2. 
$$ 
We now use the upper bound $\langle E, \bar M - \hat M_r \rangle_F \leq \| E \|_{op} \|\bar M - \hat M_r \|_*$ and the definition of $\bar M$ and $\hat M_r$ to derive
$$
\| M^* - \bar M \|_F^2 \leq \| M^* - \hat{M}_r \|_F^2 + 2 \| E \|_{op} \|\bar M - \hat M_r \|_* +  \frac{\lambda  r}{np} \sum \limits_{k>r} \sigma_k(Y)^2 - \frac{\lambda  \bar r}{np} \sum \limits_{k> \bar r} \sigma_k(Y)^2. 
$$
Let us note that we use $\sigma_k(Y) = 0$ in case $k> r_Y$. We recall that $\|\bar M - \hat M_r \|_* \leq \sqrt{r + \bar r} \cdot \|\bar M - \hat M_r \|_F$ and further obtain
\begin{align*}
    \| M^* - \bar M \|_F^2 &\leq \| M^* - \hat{M}_r \|_F^2 + 2 \| E \|_{op} \sqrt{r + \bar r} \left( \| M^* - \bar M \|_F + \| M^* - \hat M_r \|_F \right)\\
    & + \frac{\lambda r}{np} \sum \limits_{k>r} \sigma_k(Y)^2 -  \frac{\lambda \bar r}{np} \sum \limits_{k> \bar r} \sigma_k(Y)^2.
\end{align*} 
Using twice the inequality $2ab \leq \alpha a^2 + \alpha^{-1}b^2$ for $a, b >0$, with $\alpha >1$ first and with $\beta >0$ second, we get 
\begin{align}\label{eq:123}
   (1 - \alpha^{-1}) \| M^* - \bar M \|_F^2 &\leq (1 + \beta^{-1})\| M^* - \hat{M}_r \|_F^2 + (\alpha + \beta) \| E \|_{op}^2(r + \bar r)  \nonumber \\
    &+ \frac{\lambda r}{np} \sum \limits_{k>r} \sigma_k(Y)^2 -  \frac{\lambda \bar r}{np} \sum \limits_{k> \bar r} \sigma_k(Y)^2.
\end{align}

We now distinguish the two cases: $r \leq \bar r$ and $r > \bar r$. 
In the first case, namely $r \leq \bar r$, we bound from above as follows:
\begin{align*}
\frac{\lambda r}{np} \sum \limits_{k>r} \sigma_k(Y)^2 -  \frac{\lambda \bar r}{np} \sum \limits_{k> \bar r} \sigma_k(Y)^2 &
= \frac{\lambda }{np} \left( r \sum_{k=r+1}^{\bar r} \sigma_k(Y)^2 + (r-\bar r) \sum_{k >\bar r} \sigma_k(Y)^2 \right)\\
&\leq  \frac{\lambda }{np} r (\bar r - r) \sigma_{r+1}(Y)^2\\
&\leq  \frac{2\lambda r}{np} (\bar r - r) (\sigma_{r+1}(M^*)^2 + \|E\|_{op}^2) \\
& \leq \frac{2\lambda r}{np} r_{max} \sigma_{r+1}(M^*)^2  + 
\frac{2\lambda r_{max}}{np} (\bar r - r)  \|E\|_{op}^2 ,
\end{align*}
where we used Weyl inequality $\sigma_{r+1}(Y) \leq \sigma_{r+1}(M^*) + \|E\|_{op}$ leading to $\sigma_{r+1}(Y)^2 \leq 2\|E\|_{op}^2 + 2 \sigma_{r+1}(M^*)^2$. We plug this into \eqref{eq:123} to get 
\begin{align*}
   (1 - \alpha^{-1}) \| M^* - \bar M \|_F^2 &\leq (1 + \beta^{-1})\| M^* - \hat{M}_r \|_F^2 
   + \frac{2 \lambda r_{max}}{np} r \sigma_{r+1}(M^*)^2 \\
   &  + r \| E \|_{op}^2(\alpha + \beta - \frac{2 \lambda r_{max}}{np} ) \\
   &  + \bar r \|E\|_{op}^2 (\alpha + \beta + \frac{2 \lambda r_{max}}{np} ) ,
\end{align*}
for all $r \leq \bar r$ belonging to $ [r_{max}] $. Thus, for $\lambda$ such that $\frac{2 \lambda \cdot (r_{max} \vee r_Y)}{np} = (1+\varepsilon) (\alpha+\beta)$ for some $\varepsilon>0$ we get

\begin{align*}
   (1 - \alpha^{-1}) \| M^* - \bar M \|_F^2 &\leq \min_{r\in [\bar r]}\left\{ (1 + \beta^{-1})\| M^* - \hat{M}_r \|_F^2 
   + (1+\varepsilon)(\alpha + \beta) r \sigma_{r+1}(M^*)^2 \right\} \\
   & + (2+\varepsilon) (\alpha + \beta) r_{max} \|E\|_{op}^2.
\end{align*}

We now focus on the second case, namely $r > \bar r$. We observe that in this case, 
\begin{align*}
\frac{\lambda r}{np} \sum \limits_{k>r} \sigma_k(Y)^2 -  \frac{\lambda \bar r}{np} \sum \limits_{k> \bar r} \sigma_k(Y)^2 &
= \frac{\lambda }{np} \left( (r-\bar r) \sum_{k>r} \sigma_k(Y)^2 -\bar r \sum_{k =  \bar r + 1}^r \sigma_k(Y)^2 \right)\\
&\leq  \frac{\lambda (r-\bar r)}{np}  ( r_Y - r) \sigma_{r+1}(Y)^2\\
&\leq  \frac{2\lambda r}{np} r_Y  \cdot \sigma_{r+1}(M^*)^2 + \frac{2\lambda (r-\bar r)}{np} \cdot ( r_Y\vee r_{max} )  \|E\|_{op}^2 ,
\end{align*}
by a similar reasoning in the previous case. We plug this into \eqref{eq:123} to get
\begin{align*}
   (1 - \alpha^{-1}) \| M^* - \bar M \|_F^2 &\leq (1 + \beta^{-1})\| M^* - \hat{M}_r \|_F^2 
   + \frac{2 \lambda  \cdot  r_{max}\vee r_Y}{np} r \sigma_{r+1}(M^*)^2 \\
   &  + r \| E \|_{op}^2(\alpha + \beta + \frac{2 \lambda \cdot r_{max}\vee r_Y}{np} ) \\
   &  + \bar r \|E\|_{op}^2 (\alpha + \beta - \frac{2 \lambda \cdot r_{max}\vee r_Y}{np} ).
\end{align*}
With the same choice of $\lambda$ such that  $\frac{2 \lambda \cdot r_{max}\vee r_Y}{np} = (1+\varepsilon) (\alpha+\beta)$ for some $\varepsilon>0$ we get also in this case that
\begin{align*}
   (1 - \alpha^{-1}) \| M^* - \bar M \|_F^2 &\leq \min_{\bar r < r \leq r_{max}}\left\{ (1 + \beta^{-1})\| M^* - \hat{M}_r \|_F^2 
   + (1+\varepsilon)(\alpha + \beta) r \sigma_{r+1}(M^*)^2 \right\} \\
   & + (2+\varepsilon) (\alpha + \beta) r_{max} \|E\|_{op}^2.
\end{align*}

Taking $\alpha = 3/2$ and $\beta = 1/2$ and combining both cases leads to the following result
\begin{align*}
    \| M^* - \bar M \|_F^2 \leq \underset{r \in [r_{max}]}{\min} &\left\{9 \| M^* - \hat{M}_r \|_F^2 + 6(1+\epsilon) \cdot r \sigma_{r+1}(M^*)^2  \right\} + 6(2+\epsilon)\cdot r_{max} \| E \|_{op}^2 ,
\end{align*}
where we choose $\lambda$ such that $\lambda \cdot r_{max}\vee r_Y = (1+\varepsilon) np$ for some $\varepsilon>0$. 
We conclude by using the inequality \eqref{ineq: concentration}.

\subsection{Proof of Theorem~\ref{th: nuc norm predictor upper bound}}
We proceed by solving the problem in two steps for solving the optimization problem \eqref{opt: nuclear pen} which can be equivalently written as
$$
\underset{\substack{A,B\\ M = AXB}}{\min}\min_M \|Y-M\|_F^2 + 2 \lambda \cdot \|M\|_*,
$$
for $\lambda >0$. The solution to the problem in $M$ is explicit and it is known to be obtained from $Y$ by soft-thresholding of its eigenvalues: $\bar M = U_Y Diag_{n,p}((\sigma_k(Y) - \lambda)_+) V_Y^\top$, where we used the SVD of $Y$: $U_Y \Sigma_Y V_Y^\top$. Next, we project $\bar M$ on the space of matrices $AXB$ for $A$ and $B$ in Frobenius norm. It is easy to check that our choice of $\bar A, \bar B$ are exact solutions, that is $\bar M = \bar A X \bar B$. 

Similarly to the proof of Theorem~\ref{th: rank adaptive predictor upper bound}, by applying the definition of $\bar M$, expanding the squares and rearranging terms we get for all $M$:
\begin{align*}
    \|\bar M - M^*\|_F^2 & \leq \|M^* - M\|_F^2 + 2 \langle E, \bar M - M \rangle + 2 \lambda (\|M\|_* - \|\bar M\|_*)\\
    & \leq \|M^* - M\|_F^2 + 2 \sqrt{\lambda} (\| \bar M - M \|_* + \|M\|_* - \|\bar M\|_*),
\end{align*}
under the event that $\|E\|_{op}^2 \leq \lambda$. We use the decomposability of the nuclear norm of matrices as in \cite{bunea}, to find $\bar M_1$ and $\bar M_2$ such that $\bar M = \bar M_1 + \bar M_2$, $\|\bar M\|_* = \|\bar M_1\|_* + \|\bar M_2\|_*$ and $\|\bar M-M\|_* = \|\bar M_1-M\|_*+\|\bar M_2\|_*$. Moreover, $\rank(\bar M_1) \leq 2 \rank (M)$. This implies
\begin{align*}
    \|\bar M - M^*\|_F^2 & \leq \|M^* - M\|_F^2 + 4 \sqrt{\lambda} \| \bar M_1 - M \|_*\\
    & \leq \|M^* - M\|_F^2 + 4 \sqrt{\lambda} \sqrt{3 \rank (M)} \cdot \| \bar M_1 - M \|_F\\
    & \leq  \|M^* - M\|_F^2 + 4 \sqrt{\lambda} \sqrt{3 \rank (M)} \cdot (\|\bar M - M^*\|_F +\|M - M^*\|_F).
\end{align*}
We obtain for arbitrary real numbers $\alpha >1$ and $\beta >0$, for all $M$,
$$
(1 - \alpha^{-1} ) \|\bar M - M^*\|_F^2 \leq (1 +\beta^{-1}) \|M^* - M\|_F^2 + 4(\alpha+\beta) \lambda \cdot 6 \rank (M).
$$
For the particular values $\alpha = 3/2$ and $\beta = 1/2$, we get
\begin{align*}
     \|\bar M - M^*\|_F^2 &\leq  \min_M \left\{ 9 
\|M^* - M\|_F^2 + 144 \lambda \cdot \rank(M) \right\}\\
& \leq 9 \min_{r\in [n\wedge p \wedge r_X]} \left\{\min_{M: \rank M = r} \|M^* - M\|_F^2 + 16 \lambda \cdot r
\right\}.
\end{align*}
Recall that $\min_{M: \rank M = r} \|M^* - M\|_F^2 = \sum_{K=r+1}^{r^*}\sigma_K(M^*)^2 \cdot \mathbf{1}_{r < r^*} $ to get the final result.

\newpage
\section{Auxiliary results} \label{sec:algorithm}
\begin{algorithm}
\caption{Data-driven procedure for selecting $\bar{r}$ and $\lambda$}\label{alg:r_selection}
\textbf{Input:} data X, Y\\
\textbf{Require:} $np \geq (m\wedge q) r_X > 0$  \\
\textbf{Define:} $\widehat{\sigma}^2_r := \dfrac{\|Y - \hat{A}_r X \hat{B}_r \|_F^2}{np - (m\wedge q) r_X}$ \\
\textbf{Define:} $\lambda(\sigma) := 4 ({n} + {p}) \sigma^2$ \\
\textbf{Define:} $\hat{r}_\lambda := \argmin_{r \in [n \wedge p \wedge r_X]} \left( \|Y - \hat A_r X \hat B_r\|_F^2 + \lambda \cdot r \right) $ \\
\textbf{Initialize:} $r \gets r_X \wedge n \wedge p$, $\bar r \gets \hat{r}_{\lambda(\widehat{\sigma}^2_{r})}$
\begin{algorithmic}
\While{$\bar r < r$}
    \State $r \gets \bar r$
    \State $\bar r \gets \hat{r}_{\lambda(\widehat{\sigma}^2_{r})}$
\EndWhile \\
\textbf{Output:} $\bar r, \lambda(\widehat{\sigma}^2_{\bar r})$
\end{algorithmic}
\end{algorithm}

\bigskip

\noindent {\bf Acknowledgment. }{The authors thank the French National Research Agency (ANR) under the grant Labex Ecodec (ANR-11-LABEX-0047).}

\bibliographystyle{plain}
\bibliography{ref}

\end{document}